\documentclass[11pt]{article}
\usepackage{graphicx}
\usepackage{amsmath}
\usepackage{fullpage}
\usepackage{amsfonts}
\usepackage{amssymb}

\newtheorem{theorem}{Theorem}

\newtheorem{corollary}[theorem]{Corollary}

\newtheorem{lemma}[theorem]{Lemma}

\newenvironment{proof}[1][Proof]{\textbf{#1.} }{\ \rule{0.5em}{0.5em}}

\numberwithin{equation}{section}
\numberwithin{theorem}{section}

\begin{document}

\date{}
%https://arxiv.org/pdf/1612.05821.pdf

\title{ Some remarks on invariant subspaces in real Banach spaces (revised version)}

\author{ \fbox{V.I. Lomonosov} and V.S. Shulman}

\maketitle

\medskip

\medskip

{\bf Abstract.}

\medskip

It is proved that a commutative algebra $A$ of operators on a reflexive real Banach space has an invariant subspace if each operator $T\in A$ satisfies the condition
$$\|1- \varepsilon T^2\|_e \le 1 + o(\varepsilon) \text{ when  } \varepsilon\searrow 0,$$
where $\|\cdot\|_e$ is the essential norm.
This implies the existence of an invariant subspace for every commutative family of essentially selfadjoint operators on a real Hilbert space.

\medskip

\medskip

\section{Introduction }

{\bf Notations}: for a complex or real Banach space $X$ we denote by $\mathcal{B}(X)$ the algebra of all bounded linear operators on $X$; its ideals of all compact and all finite rank operators are denoted by $\mathcal{K}(X)$ and $\mathcal{F}(X)$ respectively.  The quotient $\mathcal{C}(X)= \cal{B}(X)/\cal{K}(X)$ is called the Calkin algebra; the standard epimorphism of $\mathcal{B}(X)$ onto $\mathcal{C}(X)$ is denoted by $\pi$.  The essential norm $\|T\|_e$ of $T\in \mathcal{B}(X)$ is the norm of $\pi(T)$ in $\mathcal{C}(X)$. In other words,
 $$\|T\|_e = \inf\{\|T+K\|: K\in \mathcal{K}(X)\}.$$

One of the most known unsolved problems in the invariant subspace theory in Hilbert spaces is the  existence of a (non-trivial, closed) invariant subspace for an operator $T$ with compact imaginary part; such operators are characterized by the condition $T^*-T \in \mathcal{K}(X)$ and usually called {\it essentially selfadjoint}. It is difficult to list all the papers devoted to this subject; we only mention that the answer is affirmative if  $T-T^*$ belongs to some Shatten - von Neumann class $\mathfrak{S}_p$ (Livshits \cite{MSL} for $p=1$, Sahnovich    \cite{Sah} for $p=2$, Gohberg and Krein  \cite{GK}, Macaev  \cite{M1}, Schwartz \cite{Sch} --- for arbitrary $p$), or, more generally, to the {\it Macaev ideal} (Macaev  \cite{M2} ). But the general problem is still open.  See the review \cite[Section 3]{LS2} for more information.

It was proved in \cite{Lom2} that every essentially self-adjoint operator on a complex Hilbert space has an invariant {\it real} subspace. Then in \cite{Lom3} the following  general theorem of  Burnside type was proved:
\begin{theorem}\label{L1}
  Let an algebra $A$  of operators on a real or complex Banach space  $X$  be non dense in $\mathcal{B}(X)$ with respect to the weak operator topology (WOT). Then there are non-zero $x\in X^{**}, y\in X^*$, such that
\begin{equation}\label{ineq}
|(x,T^*y)| \le \|T\|_e , \text{ for all } T\in  A.
\end{equation}
\end{theorem}

\medskip

Using this result and developing a special variational techniques, Simonic \cite{Sc} has obtained a significant progress in the topic: he proved that each essentially selfadjoint operator on a {\it real} Hilbert space has invariant subspace. Deep results based on Theorem  \ref{L1} and  variational methods were established then in papers of Atzmon \cite{A},  Atzmon, Godefroy \cite{AG}, Atzmon, Godefroy, Kalton \cite{AGK}, Grivaux \cite{SofGr} and other mathematicians. Here we will show that every commutative family of essentially selfadjoint operators on a real Hilbert space has an invariant subspace, and consider some analogs of this result for operators on Banach spaces. Our proof is very simple and short --- modulo Theorem \ref{L1}. More precisely, we use an improvement of Theorem \ref{L1} proved in \cite{Lom4} for the case of complex Banach spaces. To formulate it, we introduce some notation.

For a pair of non-zero vectors  $x\in X^{**}, y\in X^*$, let us denote by $\omega_{x,y}$ the "vector functional"\; on $\mathcal{B}(X)$, acting by the rule
$$\omega_{x,y}(T) = (x,T^*y).$$
 We say that a vector functional $\omega_{x,y}$ is {\it resolving} for an algebra $A\subset \mathcal{B}(X)$, if
 \begin{equation}\label{ineq2}
|(x,T^*y)| \le \|T\|_e (x,y), \text{ for all } T\in A.
\end{equation}
 \begin{theorem}\label{L4}
 Let a Banach algebra $A\subset \mathcal{B}(X)$  on a Banach space $X$ over $\mathbb{K}\in \{\mathbb{R}, \mathbb{C}\}$ is not WOT-dense in  $\mathcal{B}(X)$. Then at least one of the following two conditions holds:

 a) $A$ has a resolving vector functional,

b) $\mathbb{K} = \mathbb{R}$  and $A$ contains an idempotent of finite rank.
\end{theorem}
%Note that if  $A$  contains a non-zero compact operator then the vector  $x$ in (\ref{ineq2}) can be found in    $X$: indeed in this case, by \cite{Lom0}, there is non-trivial closed $A$-invariant subspace $L\subset X$ and one can take arbitrary $x\in L$, $y\in L^{\bot}$.

  %We will give a short proof of Theorem \ref{L4}, joining arguments in \cite{LSc} and \cite{Lom4}.

  As it was noted above, the "complex part"\; of Theorem \ref{L4} was proved in \cite{Lom4}. In fact a part of the proof in \cite{Lom4} worked only for reflexive $X$; in the process of our proof we correct  this inaccuracy, applying the Bishop-Phelps theorem.

  Here Theorem \ref{L4} will be deduced from the following statement:

  \begin{theorem}\label{L4N}
 Let $X$ be a real or complex Banach space. If a norm-closed algebra $A\subset \mathcal{B}(X)$
    %(is not WOT-dense in $\mathcal{B}(X)$ and --- {\bf NUZHNO LI ETO?})
     does not contain non-zero finite rank operators, then it has a resolving vector functional.
  \end{theorem}

\section{Preliminary results}

\begin{lemma}\label{squares} Let $A$ be a commutative unital algebra of operators on a real Banach space $X$. Suppose that there are non-zero $x_0\in X$, $y_0\in X^*$ such  that
$$(T^2x_0,y_0)\ge 0, \text{ for all }T\in A.$$
Then $A$ has a (closed non-trivial) invariant subspace.
\end{lemma}
\begin{proof} Let $K$ be the closed convex hull of the set
$$M = \{T^2x_0: T\in A\}\subset X.$$
Since $M$ is invariant under all operators $T^2$, $T\in A$, the same holds for $K$. Furthermore $K\neq X$, because $(z,y_0)\ge 0$, for all $z\in K$. Clearly $K \neq \{0\}$ since $A$ is unital. If $\partial K$ is a singleton $\{w_0\}$ then $w_0 =0$ because otherwise $tw_0\in \partial K$, for $t>0$. Let $w$ be a non-zero element of $K$ and let $q\in X$ be non-proportional to $w$ with $(q,y_0) \neq 0$. Then the line $w+Rq$ intersects $K$ in a point that does not belong to  $\partial K$, a contradiction.

So we may assume that $\partial K$ is not a singleton. By the Bishop-Phelps theorem \cite{BPh}, $K$ has non-zero support points. So there is a non-zero functional $y_1\in X^*$ attaining its minimum on $K$ at some non-zero point $x_1\in K$:
$$(x_1,y_1)\le (w,y_1), \text{ for all } w\in K.$$
Since $T^2x_1\in K$, for each $T\in A$, we get that
$$(x_1,y_1)\le (T^2x_1,y_1), \text{ for all } T\in A.$$
Replacing $T$ by $1+\alpha T$, $\alpha\in \mathbb{R},$
we get that
$$2\alpha(Tx_1,y_1) + \alpha^2(T^2x_1,y_1) \ge 0, \text{ for all }\alpha\in \mathbb{R}.$$
This means that $(Tx_1,y_1)= 0$, for $T\in A$, so the cyclic subspace $\overline{Ax_1}$ is a non-trivial invariant subspace for $A$.
\end{proof}

\begin{corollary} Let $T_1,T_2,...,T_n$ be commuting operators on a real Banach space $X$. If there is a positive measure $\mu$ on $\mathbf{R}^n$ and vectors $x_0\in X, y_0\in X^*$, such that
\begin{equation}\label{Atz}
(T_1^{m_1}T_2^{m_2}...T_n^{m_n}x_0,y_0) = \int_{\mathbb{R}^n}x_1^{m_1}x_2^{m_2}...x_n^{m_n}d\mu, \text{  for all  } m_i\in \mathbb{N},
\end{equation}
then operators $T_1,T_2,...,T_n$ have a common invariant subspace.
\end{corollary}
\begin{proof}
Clearly $$(P(T_1,...,T_n)x_0,y_0) = \int_{\mathbb{R}^n}P(x_1,...,x_n)d\mu, \text{ for each polynomial }P \text{ on } \mathbb{R}^n.$$
It follows that $(P(T_1,...,T_n)^2x_0,y_0)\ge 0$, for each $P$, so the algebra $A$ generated by $T_1,...,T_n$ satisfies the assumptions of Lemma \ref{squares} and, therefore,   has invariant subspaces.
\end{proof}

\begin{lemma}
\label{ComEss} Let  $T$ be an
operator on a real or complex Banach space ${X}$. If $T$ has an eigenvalue $\lambda > \rho_e(T)$  then the
norm-closed algebra $A=A(T)$ generated by $T$ contains a
non-zero finite rank operator.
\end{lemma}

\begin{proof}
Let us firstly show that the statement holds for complex spaces. Indeed by the Puncture Neighborhood Theorem (see e.g. \cite[Theorem 19.4]{Muller}) $\lambda$ is an isolated point of $\sigma(T)$, so  the Riesz
projection $P$ corresponding to $\{\lambda\}$ belongs to $A$. Thus the finite rank operator $P$ is the limit of a sequence of
polynomials in $T$.

Now turning to the case of real $X$, let ${Z}$ be the complexification of ${X}$,
i.e. ${Z} = X\oplus X$  supplied with the structure of complex space
by the formula
$$(\alpha+i\beta)z =
\alpha z + \beta Jz,$$
where  $J$ is the operator on
$Z$ acting by the rule
\[
J(x\oplus y) = (-y)\oplus x.
\]
 In matrix form
 \[ J =
\begin{pmatrix}
0 & -1\\
1 & 0
\end{pmatrix}.
\]
For each operator $K$ on $X$, the
operator $K\oplus K$ on $Z$ is complex-linear. The operator $T\oplus
T$ has the same spectrum and essential spectrum as $T$, and $\lambda$ is its eigenvalue. By the above, there is a sequence of polynomials with complex coefficients
 $$p_{n}(t) = \sum
_{k=1}^{N_{n}}(\alpha_{nk}+i\beta_{nk})t^{k}$$
such that the sequence $p_n(T\oplus T)$ tends to a non-zero finite rank operator $W$.
So
\[
p_{n}(T\oplus T) = \sum_{k=1}^{N_{n}}(\alpha_{nk}1+\beta_{nk}J)
\begin{pmatrix}
T^{k} & 0\\
0 & T^{k}%
\end{pmatrix}
= \sum_{k=1}^{N_{n}}
\begin{pmatrix}
\alpha_{nk}T^{k} & -\beta_{nk}T^{k}\\
\beta_{nk}T^{k} & \alpha_{nk}T^{k}%
\end{pmatrix}
=
\begin{pmatrix}
q_{n}(T) & r_{n}(T)\\
-r_{n}(T) & q_{n}(T)
\end{pmatrix}
,
\]
where $q_{n}$ and $r_{n}$ are polynomials with real coefficients. It follows
that
\[
W =
\begin{pmatrix}
W_{1} & -W_{2}\\
W_{2} & W_{1}%
\end{pmatrix}
\]
and $q_{n}(T) \to W_{1}$, $r_{n}(T) \to W_{2}$. Since at least one of
operators $W_{1}, W_{2}$ is non-zero, $A(T)$ contains a non-zero
finite rank operator.
\end{proof}

\bigskip

{\bf The proof of Theorem \ref{L4N}}.
%(Note  that our assumptions imply that $A^*$ is not $WOT$-dense in $\mathcal{B}(X^*)$).%
 Set
 $$F = \{T^*\in A^*: \|T\|_e < 1\}$$
  and fix $\varepsilon \in (0, \frac{1}{10})$. Suppose firstly that $Fy$ is dense in $X^*$, for each non-zero $y\in X^*$. Then the same is true for $\varepsilon Fy$. Choose $y_0\in X^*$ with $\|y_0\| = 3$ and denote by $S$  the ball $\{y: \|y-y_0\| \le 2\}$.

So for every $y\in S$, there is $T^*_y\in \varepsilon F$ with $\|T^*_yy - y_0\|<1$. By the definition of  $F$, $T_y = R_y+K_y$ where  $\|R_y\| < \varepsilon$, $K_y \in \mathcal{K}(X)$.  Thus $T^*_y = R^*_y + K^*_y$ and
$$\|K^*_yy - y_0\| \le \|T^*_yy-y_0\| + \|R^*_yy\| <  1 + \varepsilon\|y\| \le  1 + 5\varepsilon.$$
   Now the compactness of  $K_y$ implies that there is a neighborhood  $V_y$  of  $y$ in the (relative) weak* topology $\tau$ of $S$,  such that $\|K^*_yw - y_0\| < 1 +5\varepsilon$, for all $w\in V_y$. Therefore $\|T^*_yw-y_0\|< 1+5\varepsilon + 5\varepsilon < 2$. In other words $T^*_y$ maps $V_y$ to $S$.

The sets $V_y$, $y\in S$, form a covering of   $S$; since $S$ is $\tau$-compact there is a finite subcovering $\{V_{y_i}: 1\le i\le n\}$. Let $\{\varphi_i: 1\le i\le n\}$ be a unity partition related to the covering $\{V_{y_i}: 1\le i\le n\}$. We define a map $\Phi: S \to S$ by
$$\Phi(y) = \sum_{i=1}^n\varphi_i(y)T^*_{y_i}(y).$$
 By Tichonov's Theorem, $\Phi$ has a fixed point $z\in S$. Thus $Mz = z$, where $M = \sum_{i=1}^n\varphi_i(z)T^*_{y_i}$. Since the set $\varepsilon F$ is convex, $M\in \varepsilon F$.

It follows that 1 is an eigenvalue of $M$  exceeding $\|M\|_e$; by Lemma \ref{ComEss}, $A^*$ contains a non-zero finite-rank operator. So the same is true for $A$.

It remains to consider the case that $Fz$ is not dense in  $X^*$,  for some $z\in X^*$. Clearly we may assume that $Fz \neq \{0\}$, because in this case $A^*z = \{0\}$ whence $\omega_{x,z}$ is a resolving functional, for each $x\in  X^{**}$.

 Let $V=\overline{Fz}$. It is a closed convex proper subset of $X^*$, so by the Bishop - Phelps Theorem  \cite{BPh}, there are  $0\neq y \in V$ and $0\neq x \in X^{**}$ with $(x,y) = \sup\{(x,w): w\in V\}$   (in the case of complex scalars it is important that $e^{it}V = V$, for all $t\in \mathbb{R}$, see  \cite{BPh1}).  Since $F^2\subset F$, we have that $Fy\subset V$, so $(x,T^*y) \le (x,y)$, for $T^*\in F$. Therefore $(x,T^*y)\le \|T\|_e(x,y)$, for all $T\in A$.
$\blacksquare$

\medskip

%Let us call the vector functional $\omega_{x,y}$, where $x\in X^{**}, y\in X^*,$ \; ${\it resolving}$ for an algebra $A\subset \mathcal{B}(X)$, if it satisfies the condition (\ref{ineq2}).
%\medskip

\medskip

{\bf The proof of Theorem \ref{L4}}.  Suppose that a) does not hold: $A$ has no resolving functionals. Then, by Theorem \ref{L4N}, $A$ contains a non-zero finite rank operator $T_0$. If $A$ had an invariant subspace $L$ then for any $x\in L$, $y\in L^{\bot}$, the functional $\omega_{x,y}$ would be resolving for $A$. So $A$ is transitive. If $\mathbb{K} = \mathbb{C}$ this contradicts to the assumption that $A$ is not dense in $\mathcal{B}(H)$ (see \cite{Barnes}, or a more general result in \cite{Lom0}). Thus $\mathbb{K}=\mathbb{R}$.

Let $Y = T_0X$ and let $B$ be the restriction of the algebra $T_0AT_0$ to $Y$. Then $B$ is an irreducible algebra of operators on an $n$-dimensional real space. There is a well known classification of such algebras (see e.g. \cite{KShT}) which implies that each of them is isomorphic to either

    $M_n(\mathbb{R})$ or $M_{n/2}(\mathbb{R})\otimes \mathbb{C}$, or
    $M_{n/4}(\mathbb{R})\otimes {\mathbb{H}}$
         where $\mathbb{C}$ and $\mathbb{H}$ are respectively the real algebras of complex numbers or of quaternions. It follows that $B$ contains idempotents, so the same is true for $T_0AT_0$ and therefore for $A$.

$\blacksquare$.

\begin{corollary}\label{complex} Let $X$ be a complex Banach space. The only WOT-closed subalgebra of $\mathcal{B}(X)$ that has no resolving functionals is $\mathcal{B}(X)$ itself.
\end{corollary}

\medskip

In the case of real spaces the situation is different. It was proved in the recent work of E.Kissin, V.S.Shulman and Yu.V.Turovskii \cite{KShT} that in a real separable Hilbert space $H$ there is a continuum of pairwise non-similar weakly closed transitive algebras containing non-zero finite rank operators.

All these algebras has no resolving functionals: if $f= \omega_{x,y}$ is a resolving functional for a transitive algebra $A\subset \mathcal{B}(H)$, then, by (\ref{ineq2}), $(Tx,y) = 0$  for all $T\in A\cap \mathcal{F}(H)$. So the ideal  $J = A\cap \mathcal{F}(H)$ of $A$ has invariant subspaces and therefore the same is true for  $A$ --- a contradiction.
 %We proved that transitive algebras containing non-zero finite rank operators has no resolving functionals.

\section{Main results }

     In this section $X$  is a real Banach space (complex spaces are considered as real ones).
     %The standard epimorphism from $\mathcal{B}(X)$ to the Calkin algebra   $\mathcal{C}(X) = \mathcal{B}(X)/\mathcal{K}(X)$ is denoted by $\pi$.

   Let us say that an element $a$  of a unital real normed algebra is {\it positive}, if $$\|1-\varepsilon a\|\le 1+o(\varepsilon), \text{ for } \varepsilon\searrow 0.$$

   Furthermore $a$ is {\it real}, if $a^2$ is positive.

   To see that all selfadjoint operators on a real Hilbert space $H$ are real, note that the map $T\mapsto T_c$ from $\mathcal{B}(H)$ to $B(H_c)$ is $\mathbb{R}$-linear, involution-preserving and isometric, so if $T\in \mathcal{B}(H)$ is selfadjoint then
    $$\|1-\varepsilon T^2\| = \|1-\varepsilon (T_c)^2\| < 1, \text{ if } 0<\varepsilon<1/\|T\|^2.$$

    It is not difficult to check that Hermitian operators in complex Banach spaces (defined by the condition $\|\exp(itT)\| = 1$, for $t\in \mathbb{R}$) are real. Indeed, in this case
    $\|\cos(tT)\| = \|\Re(\exp(itT))\| \le 1$ whence
    $$\|1-\frac{t^2}{2}T^2\| = \|\cos tT\| - \sum_{n=2}^{\infty}(-1)^n\frac{t^{2n}}{n!}\| \le 1 + o(t^2);$$
    denoting $\frac{t^2}{2}$ by $\varepsilon$, we obtain that $\|1-\varepsilon T^2\|\le 1+o(\varepsilon).$
    \medskip
    In passing we obtain a similar criterion of reality for operators on real Banach spaces:
    \medskip

    if $\|\cos(tT)\|\le 1$, for all $t\in \mathbb{R}$, then $T$ is real.
    %(�������� ��������������� � $\|\cos ta \|\le 1$)
\medskip

    Clearly all involutions and all nilpotents of order 2 satisfy the latter condition, so  these operators are also real.
    %So the class of all essentially real operators is quite  wide.

     \medskip

   An operator $T\in \mathcal{B}(X)$ is {\it essentially real}, if $\pi(T)$ is a real element of the Calkin algebra (recall that by $\pi: \mathcal{B}(X)\to \mathcal{C}(X)$ we denote the natural epimorphism).

   Thus $T$ is essentially real if
   $$\|1-\varepsilon T^2\|_e \le 1+o(\varepsilon), \text{ when }\varepsilon\searrow 0$$

   If $T$ is an essentially selfadjoint operator on a real Hilbert space then it is essentially real. Indeed
   $$\|1-\varepsilon T^2\|_e = \|1-\varepsilon \pi(T)^2\| = \|1-\varepsilon \pi(T)^*\pi(T)\| = \|\pi(1-\varepsilon T^*T)\| \le \|1-\varepsilon T^*T\| \le 1,$$
   if $\varepsilon<\|T\|^{-2}$.

\begin{theorem}\label{LB} If $A$ is a commutative algebra of essentially real operators on a Banach space $X$, then there is a closed subspace of $X^*$, invariant for the algebra $A^*= \{T^*: T \in A\}$.
\end{theorem}
\begin{proof} Note that the set of all positive elements of a Banach algebra is a convex cone. Moreover this cone is closed. Indeed let $a = \lim_{n\to \infty}a_n$ where all $a_n$ are positive. If  $a$ is not positive then there is a sequence $\varepsilon_n \to 0$  and a number   $C> 0$, such that $\|1 - \varepsilon_n a\| > 1 + C\varepsilon_n$ for all  $n$. Taking $k$ with  $\|a-a_k\| < C/2$, we get that $\|1 - \varepsilon_na_k\| > 1 + C\varepsilon_n - \|a-a_k\|\varepsilon_n > 1 + (C/2)\varepsilon_n$, a contradiction to positivity of  $a_k$.

 It follows that the set of real elements is closed. Applying this to $\mathcal{C}(X)$ we see that the set of essentially real operators is closed. This allows us to assume that the algebra $A$ is closed. Obviously we may assume also that $A$ is unital.
 %In particular  $\exp(T)\in A$, for each  $T\in A$.

If $A$ contains a non-zero finite rank operator $K$ then, by commutativity, the finite dimensional subspace $KX$ is invariant for $A$. So we may restrict by the case that $A$ has trivial intersection with $\mathcal{F}(X)$. By Theorem  \ref{L4N}, there are vectors $x_0\in X^{**}, y_0\in X^*$, such that the condition (\ref{ineq2}) holds. Clearly if $(x_0,y_0) = 0$ then $A^*$ has an invariant subspace (namely $\overline{A^*y_0}$), so we may assume that $(x_0,y_0) >0$.

Therefore, for $T\in A$ and $\varepsilon \searrow 0$,
    $$(x_0,(1-\varepsilon (T^2)^*)y_0) \le \|1-\varepsilon T^2\|_e(x_0,y_0) \le (1+o(\varepsilon))(x_0,y_0),$$
whence
$$-\varepsilon(x_0, (T^2)^*y_0)  \le o(\varepsilon)(x_0, y_0).$$

 It follows that $(x_0,(T^2)^*y_0)\ge 0$, because $(x_0,y_0)> 0$.
Now it remains to apply Lemma \ref{squares} to the algebra $A^*$.
\end{proof}

\begin{corollary}\label{LBref} A commutative algebra of essentially real operators on a reflexive real Banach space has an invariant subspace.
\end{corollary}

Since the algebra generated by a commutative family of essentially selfadjoint operators on a Hilbert space consists of essentially selfadjoint (and therefore, essentially real) operators we get the following result:

\begin{theorem}\label{M} Any commutative family of essentially selfadjoint operators on a real Hilbert space has an invariant subspace.
\end{theorem}

It seems desirable to strengthen Theorem \ref{M} by replacing commuting essentially selfadjoint operators with essentially commuting essentially selfadjoint, or at least essentially commuting selfadjoint ones. But this is impossible:  any two selfadjoint compact operators essentially commute but not necessarily have a common invariant subspaces.

%\begin{corollary}\label{perturb} Let $N$ be a normal operator on a {\it complex} Hilbert space. Each compact perturbation of $N$ has an invariant {\it real} subspace.
%\end{corollary}
%\begin{proof} Let $ W = N+K$, where $K \in \mathcal{K}(H)$. Then $W = U+iV$, $K = P+iQ$, where operators $U,V,P,Q$ are self-adjoint. Thus operators $T_1 = U-P$ and $T_2 = V - Q$ are commuting selfadjoint operators. By corollary
\medskip

Returning to individual criteria of non-transitivity let us denote by $E(X)$  the class of all operators $T$ on a real Banach space $X$ such that each  polynomial $P(T)$ of $T$ is essentially real.

\begin{corollary}\label{LB2} If $T\in E(X)$ then the operator $T^*$ has an invariant subspace. As a consequence, if $X$ is reflexive then any operator $T\in E(X)$ has an invariant subspace.
\end{corollary}

Atzmon, Godefroy and Kalton \cite{AGK}  introduced the class $S(X)$ of all operators on $X$, satisfying the condition
\begin{equation}\label{S}
\|P(T)\|_e \le \sup\{|P(t)|: t\in \Omega\}, \text{ for each polynomial } P,
\end{equation}
 where $\Omega$ is a compact subset of $\mathbb{R}$. It was proved   in \cite{AGK} that  all operators in $(S)$ have invariant subspaces if $X$ is reflexive (in general their adjoints have invariant subspaces). The usefulness of this result was cogently demonstrated by S. Grivaux \cite{SofGr} who applied it to the proof of non-transitivity of a wide class of tridiagonal operators in sequence spaces.

 It is not difficult to see that $S(X)\subset E(X)$.  Indeed if $T\in S(X)$ then
 $$\|1-\varepsilon P(T)^2\|_e \le \sup\{|1-\varepsilon P(t)^2|: t\in \Omega\} \le 1, \text{ for sufficiently small } \varepsilon.$$
 So the above mentioned result of \cite{AGK} follows from Corollary \ref{LB2}.

\bigskip

\section{Concluding remarks and problems}

This version of the paper differs from the first one (\cite{LS1}) in several ways.  Among several improvements and additions we mention that Theorem \ref{L4} now is obtained as a consequence of Theorem \ref{L4N} which has a more convenient formulation which does not differ the real and the complex fields of scalars. Furthermore in the proof of our main result, Theorem \ref{M}, we use the quadratic function instead of the exponential one (this trick  has appeared in our review \cite[Proof of Theorem 3.9]{LS2}; the same idea was proposed two years later by Godefroy \cite{G}). Here we present it separately in Lemma \ref{squares}.

Note also that the whole text is much more detailed and (hopefully) transparent than \cite{LS1}.

\bigskip

Now we will formulate several unsolved problems.

\bigskip

Q1. (the main question):  can the word {\it real} in Theorem \ref{M} be replaced by {\it complex}?
\bigskip

Q2. Does every essentially selfadjoint operator $T$ on a real Hilbert space $H$ have {\it hyperinvariant} subspace (that is a subspace invariant for all operators commuting with $T$?
\bigskip

 Note that the positive answer for Q2 would imply the positive answer for Q1. Indeed let $H$ be complex, $H_r$ the underlying real Hilbert space and $J$  the operator of multiplication by $i$ on $H_r$. An essentially selfadjoint operator $T$ on $H$ can be considered as an essentially selfadjoint operator $T_r$ on $H_r$ that commute with $J$. A hyperinvariant subspace $L$ of $T_r$ must be invariant for $T_r$ and $J$, so it is a subspace of $H$, invariant for $T$.
 \bigskip

Q3.  Does every essentially selfadjoint operator $T$ on a real Hilbert space $H$ have {\it conditionally hyperinvariant} subspace (that is a subspace invariant for all {\it essentially selfadjoint} operators commuting with $T$)?
\bigskip

Q4. Does every essentially normal operator on real Hilbert space have invariant subspace?

A more weak version: the same question for a compact perturbation of a normal operator.

\bigskip

Q5. Does every essentially skew self-adjoint operator ($T+T^*\in \mathcal{K}(H)$) on a real Hilbert space have invariant subspace?

\medskip

\medskip

\noindent Dept of Math.\\
\noindent Kent State University\\
\noindent Kent OH 44242, USA \\
lomonoso@mcs.kent.edu

\medskip

\noindent Dept of Math.\\
\noindent Vologda State University\\
\noindent Vologda 160000, Russia\\
shulman.victor80@gmail.com

\end{document}